\documentclass{amsart}

\usepackage{amsmath}
\usepackage{amssymb}
\usepackage{amsfonts}
\usepackage{amsopn}
\usepackage{amsthm}
\usepackage[all]{xy}

\swapnumbers
\theoremstyle{plain}
\newtheorem{thm}{Theorem}[section]
\newtheorem{lem}[thm]{Lemma}

\newtheorem{prop}[thm]{Proposition}

\theoremstyle{definition}
\newtheorem{ntt}[thm]{}
\newtheorem{ex}[thm]{Example}
\newtheorem{rem}[thm]{Remark}
\newtheorem{dfn}[thm]{Definition}

\newcommand{\ii}{\mathsf{i}}            
\newcommand{\BB}{\mathfrak{B}}     
\newcommand{\GG}{\mathfrak{G}}    
\newcommand{\res}{\mathop{res}}   
\newcommand{\LL}{\mathcal{L}}       
\newcommand{\zz}{\mathbb{Z}}       
\newcommand{\qq}{\mathbb{Q}}      
\newcommand{\cc}{\mathfrak{c}}     
\DeclareMathOperator{\ind}{\mathrm{ind}}    
\DeclareMathOperator{\CH}{\mathrm{CH}}    

\begin{document}

\title{Twisted $\gamma$-filtration of a linear algebraic group.}

\author{Kirill Zainoulline}



\begin{abstract}
In the present notes we introduce and study the twisted $\gamma$-filtration on $K_0(G_s)$,
where $G_s$ is a split simple linear algebraic group over a field $k$ of characteristic prime
to the order of the center of $G_s$. We apply this filtration to construct torsion elements
in the $\gamma$-ring of the variety of complete $G$-flags, where $G$ is an inner form of $G_s$.
\end{abstract}

\maketitle


\section{Introduction}

Let $X$ be a smooth projective variety
over a field $k$. Consider the Grothendieck $\gamma$-filtration on $K_0(X)$. 
It is given by the
ideals \cite[\S2.3]{SGA6} (see also \cite[\S2]{Ka98})
$$
\gamma^iK_0(X)=
\langle c_{n_1}(b_1)\cdot \ldots \cdot c_{n_m}(b_m) \mid
n_1+\ldots + n_m\ge i,\; b_1,\ldots,b_m\in K_0(X)\rangle,\; i\ge 0
$$
generated by products of Chern classes in $K_0$.
Let
$\gamma^{i}(X)$ denote the $i$-th subsequent quotient and let 
$\gamma^*(X)=\oplus_{i\ge 0}\gamma^{i}(X)$ denote the associated graded commutative ring 
called the $\gamma$-ring of $X$.

The ring $\gamma^*(X)$  was invented by Grothendieck to approximate the topological filtration on $K_0$ and, hence,
the Chow ring $\CH^*(X)$ 
of algebraic cycles modulo rational equivalence.
Indeed,  by the Riemann-Roch theorem (see \cite[\S2]{SGA6})
the $i$-th Chern class $c_i$ induces an isomorphism with 
$\qq$-coefficients, i.e.
$c_i\colon \gamma^i(X;\qq)\stackrel{\simeq}\to \CH^i(X;\qq)$.
Moreover, in some cases the ring $\gamma^*(X)$ can be used to compute $\CH^*(X)$, 
e.g. $\gamma^1(X)=\CH^1(X)$ and there is a surjection
$\gamma^2(X)\twoheadrightarrow \CH^2(X)$ (see \cite[Ex.~15.3.6]{Fu}).

In the present notes we provide 
a uniform lower bound for the torsion part of $\gamma^*(X)$, 
where $X={}_\xi \BB_s$ is a twisted form of the variety of Borel subgoups $\BB_s$ of a split simple
linear algebraic group $G_s$ by means of a cocycle $\xi\in H^1(k,G_s)$. 
Note that the groups $\gamma^2(X)$ and $\gamma^3(X)$ 
have been studied for $G_s=PGL_n$  in \cite{Ka98} and for strongly inner forms  in \cite{GZ10}. In particular,
it was shown in \cite[\S3,7]{GZ10} that  in the strongly inner case the torsion part of $\gamma^2(X)$ determines the Rost invariant.

Our main tool is the twisted $\gamma$-filtration on $K_0(G_s)$,
where $G_s$ is a split simple linear algebraic group.
Roughly speaking, it is defined to be the image (see Definition~\ref{twfilt})
of the $\gamma$-filtration on $K_0$ of the twisted form $X$
under the composition $K_0(X)\to K_0(\BB_s)\to K_0(G_s)$, where 
the first map is given by the restriction and the second map is induced by taking the quotient. 

Let $\gamma_\xi^*$ denotes the associated graded ring of the twisted $\gamma$-filtration.
It has the following important properties: 
\begin{itemize}
\item[(i)] The ring $\gamma_\xi^*$ can be explicitly computed (see Theorem~\ref{twfiltcomp}). 
Observe that $\gamma_\xi^0=\zz$, $\gamma_\xi^1=0$ and
$\gamma_\xi^i$ is torsion for $i>1$.
\item[(ii)] There is
a surjective ring homomorphism $\gamma^*(X)\twoheadrightarrow \gamma_\xi^*$. 
Hence, $\gamma_\xi^*$ gives a lower bound for the $\gamma$-ring of the twisted form $X={}_\xi \BB_s$.
\item[(iii)] The assignment $\xi\mapsto \gamma_\xi^*$ respects the base change 
and, therefore, is an  invariant of a $G_s$-torsor $\xi$, moreover, 
the ring $\gamma_\xi^*$ can be viewed as a substitute for the $\gamma$-ring 
of the inner group  ${}_\xi G_s$.

\end{itemize}
In the last section we use these properties to construct nontrivial torsion elements
in $\gamma^2(X)$ for some twisted flag varieties $X$ (see \ref{hspinex} and \ref{e7ex}).
In particular, we establish the connection between the indexes of the Tits algebras of $\xi$ and the order
of the special cycle $\theta \in \gamma^2(X)$ constructed in \cite{GZ10}.


\section{Preliminaries.}

In the present section we recall several basic facts concerning linear algebraic groups, characters
and the Grothendieck $K_0$ (see \cite[\S24]{BI}, \cite[\S1B,\S6]{GZ10}).

\begin{ntt} Let $G_s$ be a split simple linear algebraic group of rank $n$
over a field $k$. We assume that characteristic of $k$ is prime to the order of the center of $G_s$.
We fix a split maximal torus $T$
and a Borel subgroup $B$ such that $T\subset B\subset G_s$.

Let $\Lambda_r$ and $\Lambda$ be the root and the weight lattices
of the root system of $G_s$ with respect to $T\subset B$. 
Let $\{\alpha_1,\ldots,\alpha_n\}$ be a set of simple roots
(a basis of $\Lambda_r$)
and let 
$\{\omega_1,\ldots,\omega_n\}$ be the respective 
set of fundamental weights (a basis of $\Lambda$), i.e.
$\alpha_i^\vee(\omega_j)=\delta_{ij}$.
The group of characters $T^*$ of $T$ is an intermediate lattice
$\Lambda_r\subset T^*\subset \Lambda$
that determines the isogeny class of $G_s$.
If $T^*=\Lambda$, then the group $G_s$ is
simply connected and if $T^*=\Lambda_r$ it is adjoint.
\end{ntt}

\begin{ntt}
Let $\mathbb{Z}[T^*]$ be the integral group ring of $T^*$. 
Its elements are finite linear combinations 
$\sum_i a_ie^{\lambda_i}$, $\lambda_i\in T^*$. 
Let $\BB_s$ denote the variety of Borel subgroups $G_s/B$ of $G_s$.
Consider the characteristic map for $K_0$ (see \cite[\S2.8]{De74})
$$
\cc\colon \zz[T^*] \to K_0(\BB_s)
$$
defined by sending $e^\lambda$, $\lambda\in T^*$, to the class
of the associated line bundle $[\LL(\lambda)]$.
Observe that the ring $K_0(\BB_s)$ 
does not depend on the isogeny class of $G_s$
while the group of characters $T^*$ and, 
hence, the image of $\cc$ does. 

Since $K_0(\BB_s)$ is generated by the classes $[\LL(\omega_i)]$, $i=1\ldots n$,
the characteristic map $\cc$ is surjective if $G_s$ is simply connected.
If $G_s$ is adjoint, then the image of $\cc$ is generated by the classes
$[\LL(\alpha_i)]$, where
$$
\alpha_i=\sum_j c_{ij}\omega_j\;\text{ and, therefore, }\;\LL(\alpha_i)=\otimes_j \LL(\omega_j)^{\otimes c_{ij}},
$$
and $c_{ij}=\alpha^\vee_i(\alpha_j)$ are the coefficients of the Cartan matrix of $G_s$.
\end{ntt}

\begin{ntt}
The Weyl group $W$ of $G_s$ acts on weights
via simple reflections $s_{\alpha_i}$ as
$$
s_{\alpha_i}(\lambda)=\lambda - \alpha_i^\vee(\lambda)\alpha_i,\quad \lambda\in\Lambda.
$$
For each element $w\in W$ we define (cf. \cite[\S2.1]{St75})
the weight $\rho_w \in \Lambda$ as
$$
\rho_w=\sum_{\{i\in 1\ldots n \mid w^{-1}(\alpha_i)<0\}} w^{-1}(\omega_i).
$$
In particular,
for a simple reflection $w=s_{\alpha_j}$ we have
$$
\rho_w=\sum_{\{i\in 1\ldots n \mid s_{\alpha_j}(\alpha_i)<0\}} s_{\alpha_j}(\omega_i)=
s_{\alpha_j}(\omega_j)=
\omega_j-\alpha_j.
$$
\end{ntt}

Observe that the quotient $\Lambda/\Lambda_r$ 
coincides with the group of characters of the center of the 
simply connected cover of $G_s$.
Since $W$ acts trivially on $\Lambda/\Lambda_r$, we have
$$
\bar\rho_w=\sum_{\{i\in 1\ldots n \mid w^{-1}(\alpha_i)<0\}} \bar\omega_i\quad \in \Lambda/T^*,
$$
where $\bar \rho_w$ denotes the class of $\rho_w\in \Lambda$ modulo $T^*$. In particular, $\bar \omega_i=\bar\rho_{s_{\alpha_i}}$.

\begin{ntt}
Let $\zz[\Lambda]^W$ denote the subring of $W$-invariant elements.
Then the integral group ring $\zz[\Lambda]$ is a free $\zz[\Lambda]^W$-module
with the basis $\{e^{\rho_w}\}_{w\in W}$ (see \cite[Thm.2.2]{St75}). 
Now let 
$\epsilon\colon\zz[\Lambda]\to \zz,\; e^{\lambda}\mapsto 1$
be the augmentation map.
By the Chevalley Theorem the kernel of the surjection $\cc$ is generated
by elements $x\in \zz[\Lambda]^W$ such that $\epsilon(x)=0$. Hence,
there is an isomorphism
$$\zz[\Lambda]\otimes_{\zz[\Lambda]^W}\zz\simeq \zz[\Lambda]/\ker(\cc)\simeq 
K_0(\BB_s).$$
So the elements
$$
\{g_w=\cc(e^{\rho_w})=[\LL(\rho_w)]\}_{w\in W}
$$ 
form a $\mathbb{Z}$-basis
of $K_0(\BB_s)$ called the Steinberg basis.
\end{ntt}

\begin{ntt}\label{titsalg}
Following \cite{Ti71} we associate
with each $\chi \in \Lambda/T^*$ and each cocycle $\xi\in Z^1(k,G_s)$ the central simple algebra $A_{\chi,\xi}$
over $k$ called the Tits algebra. 
This defines a group homomorphism 
$$
\beta_\xi \colon \Lambda/T^* \to Br(k)\text{  with } 
\beta_\xi(\chi)=[A_{\chi,\xi}].
$$
Let $\BB={}_\xi \BB_s$ denote the twisted form of the variety of Borel
subgroups $\BB_s$ by means of $\xi$. 
Consider the restriction map on $K_0$ over the separable closure $k_{sep}$
$$
res\colon K_0(\BB)\to K_0(\BB\times_k k_{sep})=K_0(\BB_s),
$$
where we identify $K_0(\BB\times_k k_{sep})$ with $K_0(\BB_s)$.
By \cite[Thm.4.2]{Pa94}
the image of the restriction can be identified with the sublattice 
$$
\langle \imath_w\cdot g_w\rangle_{w\in W},
$$
where $g_w=[\LL(\rho_w)]$ is an element of the Steinberg basis and 
$\imath_w=\ind(\beta_\xi(\bar\rho_w))$ is the index of the respective Tits algebra.
Observe that if $G_s$ is simply connected, then all indexes $\imath_w$
are trivial and
the restriction map becomes an isomorphism.
\end{ntt}


\section{The $K_0$ of a split simple (adjoint) group}

In the present section we provide an explicit description of the ring $K_0(G_s)$ in terms of generators and relations 
for every simple split linear algebraic group $G_s$. 
The method to compute $K_0(G_s)$ was known before, however,  due to the lack of precise references we provide the computations here. 

\begin{dfn}\label{K0def} Let $\cc\colon \zz[\Lambda]\to K_0(\BB_s)$
be the characteristic map for the simply connected cover of $G_s$. We define the ring
$\GG_s$ to be the quotient
$$
\GG_s:=\zz[\Lambda/T^*]/\overline{(\ker \cc)}
$$
and the surjective ring homomorphism $q$ to be the composite
$$
\xymatrix{
q\colon K_0(\BB_s) \ar[r]^-{\cc^{-1}}_-{\simeq}& \zz[\Lambda]/(\ker\cc) \ar@{>>}[r] & \zz[\Lambda/T^*]/\overline{(\ker c)}=\GG_s.
}
$$
\end{dfn}

Observe that if $G_s$ is simply connected, then $\GG_s=\zz$.

\begin{rem}\label{kog} By \cite[Cor.33]{Me05} applied to $X=G_s$ and to the simply-connected cover $G=\hat G_s$
of $G_s$, there is an isomorphism
$$
K_0(G_s)\simeq \zz\otimes_{R(\hat G_s)} K_0(\hat G_s,G_s),
$$
where $R(\hat G_s)\simeq \zz[\Lambda]^W$ is the representation ring. By \cite[Cor.5]{Me05} applied to 
$G=\hat G_s$, $X=Spec\; k$ and $G/H=G_s$ there is an isomorphism
$$K_0(\hat G_s,G_s)\simeq R(H),$$ 
where $R(H)\simeq \zz[\Lambda/T^*]$ is the representation ring.
Therefore,
$$
K_0(G_s)\simeq \zz\otimes_{\zz[\Lambda]^W} \zz[\Lambda/T^*]\simeq \GG_s.
$$
\end{rem}

\begin{lem}\label{iddescr} The ideal $\overline{(\ker\cc)}\subset \zz[\Lambda/T^*]$ 
is generated by the elements
$$d_i(1-e^{\bar\omega_i}),\; i=1\ldots n,$$ 
where $d_i$ is the dimension of the $i$-th fundamental
representation. 
\end{lem}

\begin{proof}
By the Chevalley Theorem the subring of invariants $\zz[\Lambda]^W$ can be identified with the
polynomial ring $\zz[\rho_1,\ldots,\rho_n]$, where $\rho_i$ is the $i$-th
fundamental representation, i.e.
$$
\rho_i=\sum_{\lambda \in W(\omega_i)}e^\lambda
$$
(here $W(\omega_i)$ denotes the $W$-orbit of the fundamental weight $\omega_i$).

Since $d_i=\epsilon(\rho_i)$, $\ker\cc=(d_1-\rho_1,\ldots,d_n-\rho_n)$.
To finish the proof observe that $\overline{(d_i-\rho_i)}=d_i(1-e^{\bar\omega_i})$.
\end{proof}

\begin{rem}Observe that by definition and \ref{iddescr} we have $\GG_s\otimes\qq\simeq \qq$. \end{rem}

\begin{ntt} In the following examples we compute the ring $\GG_s\simeq K_0(G_s)$ 
for every simple split linear algebraic group $G_s$
(we refer to \cite[\S24]{BI} for the description of $\Lambda/T^*$ and to \cite[Ch.8, Table~2]{Bo} for the dimensions of fundamental representations).

\vspace{1ex}
\noindent
\begin{tabular}{l|l|l}
$\Lambda/T^*$ & $G_s$, $m\ge 1$ & Example \\ \hline
$\zz/m\zz$, $m\ge 2$ & $SL_{n+1}/\mu_m$ & \eqref{exan} \\
$\zz/2\zz$ & $O^+_{m+4}$, $PSp_{2m+2}$, $HSpin_{4m+4}$, $E_7^{ad}$ & \eqref{gex}\\
$\zz/2\zz\oplus\zz/2\zz$ & $PGO^+_{4m+4}$ & \eqref{pgo} \\
$\zz/3\zz$ & $E_6^{ad}$ & \eqref{sl3} \\
$\zz/4\zz$ & $PGO^+_{4m+2}$ & \eqref{pgoodd}
\end{tabular}

\end{ntt}

\begin{ex}\label{exan} 
Consider the case $G_s=SL_{n+1}/\mu_m$, $m\ge 2$. 
The group $G_s$ has type $A_n$ and
$\Lambda/T^*=\langle \sigma\rangle$ is cyclic of order $m$. 
The quotient map $\Lambda/\Lambda_r\to \Lambda/T^*$ sends 
$\bar \omega_i\in \Lambda/\Lambda_r$, $i=1\ldots n$ to 
$(i\; {\rm mod}\; m) \sigma\in  \Lambda/T^*$.

By Definition~\ref{K0def} and Lemma~\ref{iddescr} we have
$$\GG_s\simeq\zz[y]/(1-(1-y)^m, d_1y, \ldots ,d_{m-1}y^{m-1}),$$ 
where $y=(1-e^\sigma)$ and  
$d_j=gcd\{\tbinom{n+1}{i}\mid i\equiv j \mod m,\; i=1\ldots n\}$.

In particular, for $G_s=SL_p/\mu_p=PGL_p$, where $p$ is a prime, we obtain
$$
\GG_s\simeq \zz[y]/(\tbinom{p}{1}y,\tbinom{p}{2}y^2,\ldots,\tbinom{p}{p-1}y^{p-1},y^p).
$$
\end{ex}

\begin{ex}\label{gex}
Assume that $\Lambda/T^*=\langle\sigma\rangle$ has 
order $2$. Then 
$$\GG_s\simeq\zz[y]/(y^2-2y, dy),$$ 
where $y=(1-e^\sigma)$ and $d$ is the g.c.d. of dimensions of representations
corresponding to $\omega_i$ with $\bar\omega_i=\sigma$. 
The integer $d$ can be determined as follows:

\vspace{1ex}
\paragraph{$B_n$:} We have $\Lambda/\Lambda_r=\{0,\bar\omega_n\}\simeq \zz/2\zz$ which corresponds to the
adjoint group $G_s=O^+_{2n+1}$. Since  $\bar\omega_i=0$ for each $i\neq n$, $d$ coincides with the dimension of $\omega_n$
that is $2^n$.

\vspace{1ex}
\paragraph{$C_n$:} We have $\Lambda/\Lambda_r=\{0,\sigma=\bar\omega_1=\bar\omega_3=\ldots\}\simeq \zz/2\zz$ that is $G_s=PSp_{2n}$.
Since $\bar\omega_i=0$ for even $i$, $d$ is the g.c.d. of dimensions of $\omega_1,\omega_3,\ldots$, i.e.
$$d=gcd\big(2n,\tbinom{2n}{3}-\tbinom{2n}{1},\tbinom{2n}{5}-\tbinom{2n}{3},\ldots\big).$$ 

\paragraph{$D_n$:} 
If $n$ is odd, then
$\Lambda/\Lambda_r=\{0,\bar\omega_{n-1},\bar\omega_1,\bar\omega_n\}\simeq \zz/4\zz$, where
$\bar\omega_1=2\bar\omega_{n-1}=2\bar\omega_n$.
Therefore, $\Lambda/T^*\simeq \zz/2\zz$ if it is a quotient of $\Lambda/\Lambda_r$ modulo the subgroup $\{0,\bar \omega_1\}$. In this case $\Lambda/T^*=\{0,\sigma=\bar\omega_{n-1}=\bar\omega_n\}$ which corresponds to
the special orthogonal group $G_s=O^+_{2n}$. Since $\bar\omega_s=s\bar \omega_1$ for $2\le s\le n-2$ and
$\bar\omega_1=0$ in $\Lambda/T^*$, $d$ is the g.c.d. of the dimensions of $\omega_{n-1}$ and $\omega_n$
that is $2^{n-1}$.

\vspace{1ex}
\noindent
If $n$ is even, then
$\Lambda/\Lambda_r=\{0,\bar\omega_{n-1}\}\oplus\{0,\bar\omega_n\}\simeq\zz/2\zz\oplus \zz/2\zz$, where $\bar\omega_1=\bar\omega_{n-1}+ \bar\omega_n$.
In this case, we have two cases for $\Lambda/T^*$:
\begin{enumerate}
\item It is the quotient of $\Lambda/\Lambda_r$
modulo the diagonal subgroup
$\{0,\bar\omega_{n-1}+\bar\omega_n\}$.
Then $\Lambda/T^*=\{0,\sigma=\bar\omega_{n-1}=\bar\omega_n\}$, $G_s=O^+_{2n}$ and $d$ is the same as in the odd case, i.e. $d=2^{n-1}$. 

\item It is the quotient modulo one of the factors, e.g.
$\Lambda/T^*=\{0,\sigma=\bar\omega_{n-1}\}$, where $\bar\omega_n=0$. This corresponds to the half-spin group $G_s=HSpin_{2n}$. We have
$\bar\omega_1=\bar\omega_3=\ldots=\bar\omega_{n-1}$ and $\bar\omega_i=0$ if $i$ is even. 

\noindent
Therefore,
$d=gcd\big(2n,\tbinom{2n}{3},\ldots,\tbinom{2n}{n-3},2^{n-1}\big)$
which implies that $d=2^{v_2(n)+1}$, where $v_2(n)$ denotes the $2$-adic valuation of $n$.
\end{enumerate}

\vspace{1ex}
\paragraph{$E_7$:} We have $\Lambda/\Lambda_r=\{0,\sigma=\bar\omega_7=\bar\omega_5=\bar\omega_2\}\simeq \zz/2\zz$ with $\bar\omega_1=\bar\omega_3=\bar\omega_4=\bar\omega_6=0$.
Therefore, $d=gcd(56, \tbinom{56}{3},912)=8$.
\end{ex}

\begin{ex}\label{pgo}
Assume that $\Lambda/T^*=\langle\sigma_1\rangle \oplus \langle \sigma_2\rangle$, where $\sigma_1$ and $\sigma_2$
are of order 2. 
In this case $G_s=PGO^+_{2n}$ is an adjoint group ($T^*=\Lambda_r$) of type $D_n$ with $n$ even. 
We have $\sigma_1=\bar\omega_{n-1}$ and $\sigma_2=\bar\omega_n$, $\bar\omega_s=s\bar\omega_1$, $2\le s\le n-2$, $2\bar\omega_1=0$ and $\bar\omega_1=\bar\omega_{n-1}+\bar\omega_n$.
Then
$$
\GG_s\simeq \zz[y_1,y_2]/(y_1^2-2y_1, y_2^2-2y_2,d_1y_1,d_2y_2,d(y_1+y_2-y_1y_2)),
$$
where $y_1=(1-e^{\sigma_1})$, $y_2=(1-e^{\sigma_2})$;  $d_1$ (resp. $d_2$) is the g.c.d. of dimensions of $\omega_i$ with $\bar\omega_i=\bar\omega_{n-1}$ (resp. $\bar\omega_i=\bar\omega_n$) that is $d_1=d_2=2^{n-1}$; and $d$ is the g.c.d. of dimensions
of $\omega_1,\omega_3,\ldots,\omega_{n-3}$ that is
$d=gcd\big(2n,\tbinom{2n}{3},\ldots \tbinom{2n}{n-3}\big)$.

In particular, for $G_s=PGO^+_8$ we obtain
$$
\GG_s\simeq \zz[y_1,y_2]/(y_1^2-2y_1,y_2^2-2y_2,8y_1,8y_2,8y_1y_2).
$$
\end{ex}

\begin{ex}\label{sl3}
Assume that $\Lambda/T^*=\langle\sigma\rangle$ has 
order $3$.
Then 
$$\GG_s\simeq\zz[y]/(y^3-3y^2+3y, d_1y, d_2y^2),$$ 
where $y=(1-e^\sigma)$ and
$d_1$ (resp. $d_2$) is the greatest common divisor of dimensions of fundamental representations $\omega_i$, $i=1\ldots n$
such that $\bar\omega_i=\sigma$ (resp. $\bar\omega_i=2\sigma$).

For the adjoint group of type $E_6$ we have $\Lambda/\Lambda_r=\{0,\sigma=\bar\omega_1=\bar\omega_5,2\sigma=\bar\omega_2=\bar\omega_6\}$ with $\bar\omega_2=\bar\omega_4=0$.
Therefore, $d_1=d_2=gcd(27,\tbinom{27}{2})=27$.
\end{ex}

\begin{ex}\label{pgoodd} Assume that $\Lambda/T^*=\langle\sigma\rangle$ has 
order $4$.
Then 
$$\GG_s\simeq\zz[y]/(y^4-4y^3+6y^2-4y, d_1y, d_2y^2, d_3y^3),$$ 
where $y=(1-e^\sigma)$. For the group $PGO^+_{2n}$ where $n$ is odd we have $\sigma=\bar\omega_{n-1}$,
$2\sigma=\bar\omega_1$ and $3\sigma=\bar\omega_n$. Therefore, $d_1=d_3=2^{n-1}$ and 
$d_2=gcd(\tbinom{2n}{1},\tbinom{2n}{3},\ldots,\tbinom{2n}{n-2})$.
\end{ex}


\section{The twisted $\gamma$-filtration.}

In the present section we introduce and study the twisted $\gamma$-filtration.

\begin{ntt}
Let $\gamma=\ker\epsilon$ denote the augmentation ideal in $\zz[\Lambda]$. 
It is generated by the differences
$$
\langle (1-e^{-\lambda}),\; \lambda\in \Lambda \rangle.
$$
Consider the $\gamma$-adic filtration on $\zz[\Lambda]$
$$
\zz[\Lambda]=\gamma^0\supseteq \gamma\supseteq \gamma^2\supseteq \ldots
$$
The $i$-th power $\gamma^i$ is generated by products of at least $i$ differences.
\end{ntt}

\begin{dfn}\label{deffilt} We define the filtration on $K_0(\BB_s)$ (resp. on $\GG_s$) to be the image of the $\gamma$-adic
filtration on $\zz[\Lambda]$ via $\cc$ (resp. via $q$), i.e.
$$
\gamma^iK_0(\BB_s):=\cc(\gamma^i) \text{ and }\gamma^i\GG_s :=q(\gamma^iK_0(\BB_s)),\; i\ge 0.
$$
So that we have a commutative diagram of surjective group homomorphisms
$$
\xymatrix{
\gamma^i \ar[dr] \ar[r]^-{\cc}& \gamma^iK_0(\BB_s)\ar[d]^q \\
 & \gamma^i\GG_s
}
$$
\end{dfn}

\begin{lem} The $\gamma$-filtration on $K_0(\BB_s)$ coincides with the filtration
introduced in Definition~\ref{deffilt}.
\end{lem}

\begin{proof} Since $K_0(\BB_s)$ is generated by the classes of line bundles,
$$
\gamma^iK_0(\BB_s)=\langle c_1([\LL_1])\cdot\ldots \cdot c_1([\LL_m]) \mid m\ge i,\; \LL_j\in K_0(\BB_s)\rangle.
$$
Moreover, each line bundle $\LL$ is the associated bundle $\LL=\LL(\lambda)$ for some character $\lambda\in \Lambda$.
Therefore, $c_1([\LL])=1-[\LL^\vee]=\cc(1-e^{-\lambda})$ (see \cite[\S2.8]{De74}).
\end{proof}

\begin{dfn}\label{twfilt} Given a $G_s$-torsor $\xi\in H^1(k,G_s)$ and the respective twisted form $\BB={}_\xi\BB_s$ 
we define
the twisted filtration on $\GG_s$ to be the image of the $\gamma$-filtration on $K_0(\BB)$ via
the composite $\res\circ q$, i.e.
$$
\gamma_\xi^i\GG_s:=q(\res(\gamma^iK_0(\BB))),\; i\ge 0.
$$
Let $\gamma_\xi^{i/i+1}\GG_s=\gamma_\xi^i\GG_s/\gamma_\xi^{i+1}\GG_s$.
The associated graded ring 
$
\bigoplus_{i\ge 0} \gamma_\xi^{i/i+1}\GG_s
$
will be called the $\gamma$-invariant of the torsor $\xi$ and will be denoted simply as $\gamma_\xi^*$.
\end{dfn}

Note that the Chern classes commute with restrictions, therefore the restriction map $\res\colon \gamma^iK_0(\BB)\to \gamma^iK_0(\BB_s)$ is well-defined. By definition there is a surjective ring homomorphism
$$
\gamma^*(\BB)\twoheadrightarrow \gamma_\xi^*.
$$

\begin{thm}\label{twfiltcomp} The twisted filtration $\gamma_\xi^i\GG_s$ can be computed as follows:
$$
\gamma_\xi^i\GG_s=\langle 
\prod_{j=1}^m \binom{\ind(\beta_\xi(\bar\rho_{w_j}))}{n_j}(1-e^{\bar \rho_{w_j}})^{n_j}\mid n_1+\ldots + n_m \ge i,\; w_j\in W\rangle.
$$
\end{thm}

\begin{proof} Since the Chern classes commute with restrictions,
the image of the restriction $\res\colon \gamma^iK_0(\BB)\to \gamma^iK_0(\BB_s)$ is 
generated by the products
$$
\langle c_{n_1}(\imath_{w_1}g_{w_1})\cdot \ldots \cdot c_{n_m}(\imath_{w_m}g_{w_m}) 
\mid n_1+\ldots +n_m\ge i,\; w_1,\ldots,w_m\in W\rangle,
$$
where $\{\imath_{w_j}\}$ are the indexes of the respective Tits algebras from \ref{titsalg}.
Applying the Whitney formula for the Chern classes \cite[\S3.2]{Fu} we obtain
$$
c_j(\imath_w g_w)=\binom{\imath_w}{j}c_1(g_w)^j.
$$
Therefore, $q(\binom{\imath_w}{j}c_1(g_w)^j)=\binom{\imath_w}{j}(1-e^{-\bar\rho_w})^j$, where $\imath_w=\ind(\beta_\xi(\bar \rho_w))$.
\end{proof}

\begin{ex} 
Since $\gamma^0(X)\simeq \zz$ and $\gamma^1(X)=Pic(X)$ is torsion free for every smooth projective $X$, 
we obtain that $\gamma_\xi^0\simeq \zz$ and $\gamma_\xi^1=0$ for any $\xi$.
\end{ex}

\begin{ex}[Strongly-inner case] If $\beta_\xi=0$,
then $\binom{\imath_{w_j}}{n_j}=1$ and
$\gamma_\xi^{i}\GG_s=\gamma^{i}\GG_s$. 
\end{ex}

\begin{ex}[$\zz/2\zz$-case]\label{zz2case}  As in \ref{gex}
assume that $\Lambda/T^*=\langle\sigma\rangle$ has order 2 and $\beta_\xi\neq 0$.
Then there is only one non-split Tits algebra $A=A_{\sigma,\xi}$ and it has exponent $2$.
Let $\ii_A=v_2(\ind(A))$ denote
the 2-adic valuation of the index of $A$.
By definition we have
$$
\gamma_\xi^{i}\GG_s=
\langle \tbinom{2^{\ii_A}}{n_1}\ldots 
\tbinom{2^{\ii_A}}{n_m} 2^{n_1+\ldots +n_m-1}y\mid n_1+\ldots +n_m\ge i\rangle
$$
in $\zz[y]/(y^2-2y,dy)$, where $y=1-e^\sigma$ and $d$ is given in \ref{gex}. 
Observe that modulo the relation $y^2=2y$ these ideals are generated by (for $j\ge 1$)

\vspace{1ex}
\noindent
\begin{tabular}{ll}
$\gamma_\xi^{2j-1}\GG_s=\gamma_\xi^{2j}\GG_s=\langle 2^{2j-1}y\rangle$ & if $\ii_A=1$;\\
$\gamma_\xi^{4j-3}\GG_s=\gamma_\xi^{4j-2}\GG_s=\langle 2^{4j-2}y\rangle$,\;
$\gamma_\xi^{4j-1}\GG_s=\gamma_\xi^{4j}\GG_s=\langle 2^{4j-1}y\rangle$ & if $\ii_A=2$;\\
$\gamma_\xi^{1}\GG_s=\gamma_\xi^{2}\GG_s=
\langle 2^{\ii_A}y\rangle,\;
\gamma_\xi^{3}\GG_s=\gamma_\xi^{4}\GG_s=
\langle 2^{\ii_A+1}y\rangle,
\gamma_\xi^{5}\GG_s=\langle 2^{\ii_A+4}y\rangle \ldots$ & if $\ii_A>2$. 
\end{tabular}

\vspace{1ex}
\noindent
Taking these generators modulo the relation $dy=0$
we obtain the following formulas for the second quotient $\gamma_\xi^2$:
$$
\text{if }\ii_A=1,\text{ then }\gamma_\xi^2=
\begin{cases}
0 & \text{ if } v_2(d)\le 1 \\
\zz/2\zz & \text{ if }  v_2(d)=2\\
\zz/4\zz & \text{ if } v_2(d)\ge 3
\end{cases}
$$
$$
\text{if }\ii_A>1,\text{ then }\gamma_\xi^2=
\begin{cases}
0 & \text{ if } v_2(d)\le \ii_A \\
\zz/2\zz & \text{ if }  v_2(d)>\ii_A
\end{cases}
$$
\end{ex}

\begin{ex}[$\zz/2\zz\oplus\zz/2\zz$-case]\label{z2z2} As in \ref{pgo} assume that $\Lambda/T^*=\langle \sigma_1\rangle\oplus\langle\sigma_2 \rangle$, where $\sigma_1, \sigma_2$ have order 2. This is the case
for the adjoint group $PGO^+_{2n}$ where $n$ is even. 
Assume that $n=4$ which corresponds
to the group of type $D_4$, i.e. $PGO^+_8$. Let 
$C^+$ and $C^-$ denote the Tits algebras corresponding to the generators $\sigma_1=\bar\omega_3$ and $\sigma_2=\bar\omega_4$.
Let $A$ denote the Tits algebra corresponding to the sum $\sigma_1+\sigma_2$.
Note that $C^+\times C^-$ is the even part of the Clifford algebra of the algebra with involution $A$
and $[A]=[C^+\otimes C^-]$ in $Br(k)$.

By definition we have  in $\zz[y_1,y_2]$
$$
\gamma_\xi^i\GG_s=\langle \tbinom{ind\; C_+}{n_1}y_1^{n_1}\cdot \tbinom{ind\; C_-}{n_2}y_2^{n_2}\cdot\tbinom{ind\; A}{n_3}(y_1+y_2-y_1y_2)^{n_3} \mid n_1+n_2+n_3\ge i \rangle.
$$
Modulo the relations $(y_1^2-2y_1,y_2^2-2y_2,8y_1,8y_2,8y_1y_2) $ we obtain that
$$
\gamma_\xi^2\GG_s\simeq 
\frac{(ind\; C_+)\zz}{8\zz}\oplus \frac{(ind\; C_-)\zz}{8\zz} \oplus \frac{(ind\; A)\zz}{8\zz}
$$
\end{ex}


\section{Torsion in the $\gamma$-filtration.}

In the present section we show how the twisted $\gamma$-filtration can be used to construct
nontrivial torsion elements in the $\gamma$-ring of the twisted form $\BB$ of a variety of Borel subgroups.

\begin{ntt}
For simplicity we consider only the case of $G_s$ (see Examples \ref{gex} and \ref{zz2case}) with  
$\Lambda/T^*=\langle\sigma\rangle$ of order 2.
Let $d$ denote the g.c.d. of dimensions of 
fundamental representations corresponding to $\sigma$.

Given a $G_s$-torsor $\xi\in H^1(k,G_s)$ let $\ii_A$ denote
the 2-adic valuation of the index of the Tits algebra $A=A_{\sigma,\xi}$.
Let $\BB={}_\xi \BB_s$ denote the twisted form
of the variety of Borel subgroups of $G_s$ by means of $\xi$. 
Consider the respective twisted filtration $\gamma_\xi^i\GG_s$ on $\GG_s$.
\end{ntt}

\begin{prop}\label{hspincase} 
Assume that $v_2(d)>\ii_A\ge 3$.
Then for each $\lambda\in \Lambda$ such that $\bar\lambda=\sigma$
there exists a non-trivial torsion
element of order $2$ in $\gamma^2(\BB)$.
Moreover, its image in $\gamma^2_\xi=\zz/2$ (via $q$) is non-trivial
and in $\gamma^2(\BB_s)$ (via $\res$) is trivial.
\end{prop}

\begin{proof} The proof of this result was inspired 
by the proof of \cite[Prop.4.13]{Ka98}.

Let $g=[\LL(\lambda)]$ denote the class of the associated line bundle.
Using the formula for the first Chern class of a tensor product of line bundles for $K_0$ we obtain
$$
c_1(g)^2=2c_1(g)-c_1(g^2).
$$
Hence,
$$
c_1(g)^4=(2c_1(g)-c_1(g^2))^2=4c_1(g)^2-4c_1(g)c_1(g^2)+c_1(g^2)^2.
$$
Therefore,
$$
\eta=4c_1(g)^3-c_1(g)^4=4c_1(g)^2-c_1(g^2)^2\in\gamma^3K_0(\BB_s).
$$ 

We claim that the class of
$2^{\ii_A-3}\eta$ gives the desired torsion element. 

Indeed, 
$c_1(g^2)=c_1([\LL(2\lambda)])$. 
Since $2\lambda\in T^*$, $[\LL(2\lambda)] \in \cc(T^*)$ and, therefore, by \cite[Cor.3.1]{GiZ} 
$c_1(g^2)\in\gamma^1K_0(\BB)$.
Moreover, we have $2^{\ii_A-1}c_1(g)^2=c_2(2^{\ii_A}g)$, where $2^{\ii_A}g \in K_0(\BB)$.
Hence, $2^{\ii_A-1}c_1(g)^2 \in \gamma^2K_0(\BB)$.
Combining these together we obtain that $2^{\ii_A-3}\eta\in \gamma^2K_0(\BB)$.

Now since $2^{\ii_A-3}\eta\in \gamma^2K_0(\BB)$ its image in $\gamma_\xi^2\GG_s$  can be computed as
$$
q(2^{\ii_A-3}\eta)=2^{\ii_A-3}q(\eta)=2^{\ii_A-1}q(c_1(g)^2)=2^{\ii_A-1}(1-e^{-\sigma})^2=2^{\ii_A}y.
$$
But $q(2^{\ii_A-3}\eta)\notin \gamma_\xi^3\GG_s=\langle 2^{\ii_A+1}y\rangle$.
Therefore, $2^{\ii_A-3}\eta\notin \gamma^3K_0(\BB)$.

From the other hand side $2^{\ii_A-2}\eta=2^{\ii_A}c_1(g)^3+2^{\ii_A-2}c_1(g)^4$ is in $\gamma^3K_0(\BB)$.
So the class of $2^{\ii_A-3}\eta$ gives the desired torsion element of order 2.
\end{proof}

\begin{ex}\label{hspinex} Let $G_s=HSpin_{2n}$ be a half-spin group of rank $n\ge 4$.
So $G_s$ is of type $D_n$, where $n$ is even, 
$\Lambda/T^*=\langle \sigma=\bar\omega_1\rangle\simeq \zz/2\zz$ and according to Example~\ref{gex}
we have $d=2^{v_2(n)+1}$. 
Let $\xi \in H^1(k,G_s)$ be a non-trivial torsor. Then
there is only one Tits algebra $A=A_{\sigma,\xi}$; 
it has exponent $2$ and
index $2^{\ii_A}$ such that $\ii_A\le v_2(n)+1$.

Recall that each such torsor corresponds to an algebra with orthogonal involution $(A,\delta)$ with trivial
discriminant and trivial component of the Clifford algebra.
The respective twisted form $\BB={}_\xi \BB_s$ then corresponds to the variety of Borel
subgroups of the group $PGO^+(A,\delta)$.

Applying the proposition to this case we obtain that
for any such algebra $(A,\delta)$ where $8\mid \ind(A)$ and $A$ is non-division,
there exists a non-trivial torsion element of order 2 in $\gamma^2(\BB)$ that vanishes over a splitting field of $(A,\delta)$.
\end{ex}

\begin{lem}\label{elg2} The $\gamma$-filtration on $K_0(\BB_s)$ is generated by the first Chern classes 
$c_1([\LL(\omega_i)])$, $i=1\ldots n$, i.e.
$$
\gamma^iK_0(\BB_s)=
\langle \prod_{j\in 1\ldots n}c_1([\LL(\omega_j)]) \mid \text{the number of elements in the product }\ge i \rangle.
$$
In particular, the second quotient $\gamma^2(\BB_s)$ is additively generated by the products
$$
\gamma^2(\BB_s)=\langle c_1([\LL(\omega_i)])c_1([\LL(\omega_j)])\mid i,j\in 1\ldots n\rangle.
$$
\end{lem}

\begin{proof}
Each $b\in K_0(\BB_s)$ can be written
as a linear combination $b=\sum_{w\in W} a_wg_w$.
Therefore, any Chern class of $b$ can be expressed in terms of $c_1(g_w)$.

Each $\rho_w$ can be written uniquely as a linear combination of fundamental weights $\{\omega_1,\ldots,\omega_n\}$.
Therefore, by the formula for the Chern class of the  tensor product of line bundles \cite[8.2]{CPZ}, each $c_1(g_w)$ can
be expressed in terms of $c_1([\LL(\omega_i)])$.
\end{proof}

\begin{ex}\label{e7ex}
Let $G_s$ be an adjoint group of type $E_7$ and let $\xi\in H^1(k,G_s)$ be a non-trivial $G_s$-torsor.
Then there is only one nonsplit Tits algebra $A=A_{\sigma,\xi}$ of exponent 2 and $\ii_A\le 3$.
Let $\BB={}_\xi \BB_s$ be the respective twisted flag variety.

By Lemma~\ref{elg2} any element of $\gamma^2(\BB)$ can be written as
$$
x=\sum_{ij} a_{ij}c_1([\LL(\omega_i)])c_1([\LL(\omega_j)]) \in \gamma^2(\BB)
$$
for certain coefficients $a_{ij}\in \zz$.
Since $\sigma=\bar\omega_7=\bar\omega_5=\bar\omega_2$ and $\bar\omega_1=\bar\omega_3=\bar\omega_4=\bar\omega_6=0$, we obtain that
$$
q(x)=C\cdot 2y \in \gamma_\xi^2, \text{ where } C=a_{25}+a_{27}+a_{57}+a_{22}+a_{55}+a_{77}.
$$
Therefore, $q(x)\neq 0$ in $\gamma^2_\xi$ if and only if $4\nmid C$ and $\ii_A\le 2$.

Consider the class $\cc(\theta) \in \gamma^2K_0(\BB_s)$ of the special cycle 
$\theta$ constructed in \cite[Def.3.3]{GZ10}. Note that the image of $\theta$ in $CH^2(\BB)$ can be viewed as a generalization
of the Rost invariant for split adjoint groups (see Remark~\ref{cohinv}).

If $\ii_A=1$, then by \cite[Prop.6.5]{GZ10} we know that $\cc(\theta)\in \gamma^2(\BB)$ is a non-trivial torsion element.
If $\ii_A=2$, then following the proof of \cite[Prop.6.5]{GZ10} we obtain that $2\cc(\theta) \in \gamma^2(\BB)$.

We claim that if $\ii_A\le 2$, then $x=2\cc(\theta)$ is non-trivial.
Indeed, in this case $4\nmid C=a_{22}+a_{55}+a_{77}=6$, therefore, we have $q(x)\neq 0$, and $x\neq 0$ in $\gamma^2(\BB)$.
In particular, this shows that for $\ii_A=1$ the order of the special cycle $\theta$ in $\gamma^2(\BB)$ is divisible by 4.
\end{ex}

\begin{ex} Let $\xi\in H^1(k,PGO_8^+)$. Applying the same arguments as in \ref{e7ex} to Example~\ref{z2z2} 
we obtain that  if $ind(A), ind(C_+), ind(C_-)\le 4$, then  $2\cc(\theta) \in \gamma^2(\BB)$ is non-trivial.
\end{ex}

We finish by the following remark that 
provides another motivation for the study of the torsion part of $\gamma^*(\BB)$

\begin{rem}\label{cohinv} Recall that by the Riemann-Roch theorem the second Chern class induces a surjection 
$c_2\colon Tors\, \gamma^2(\BB) \twoheadrightarrow Tors \CH^2(\BB)$ \cite[Cor.2.15]{Ka98}, where the latter group is isomorphic to the cohomology quotient \cite[Thm.2.1]{Pe98}
$$\frac{\ker(H^3(k,\qq/\zz(2)) \to H^3(k(\BB),\qq/\zz(2))}{\oplus_{\chi\in \Lambda/\Lambda_r} N_{k_\chi/k}(k_\chi^*\cup \beta_\xi(\chi))},$$ 
where $k_\chi$ denotes the fixed subfield of $\chi$. Therefore, the group $Tors\, \gamma^2(\BB)$ can be viewed as an upper bound for
the group of cohomological invariants of $G_s$-torsors in degree $3$.
\end{rem}

\smallskip
\noindent{\small{\textbf{Acknowledgments.} I am grateful to Alexander Merkurjev for pointing out Remark~\ref{kog}.
This work was supported by NSERC Discovery 385795-2010 and 
Accelerator Supplement 396100-2010 grants.

\bibliographystyle{amsplain}

\end{document}